\documentclass{svmult}

\usepackage{bbm}
\usepackage{amsfonts}
\usepackage{mathrsfs}

\usepackage{amssymb,amsmath}
\usepackage{cases}
\usepackage{color}
\usepackage{enumitem}
\newenvironment{enumerate-roman}{\begin{enumerate}}{\end{enumerate}}
\usepackage{makeidx}         
\usepackage{graphicx}        
\usepackage{multicol}        
\usepackage[bottom]{footmisc}

\makeindex             

\allowdisplaybreaks[4]
\newtheorem{thm}{Theorem}[section]
\newtheorem{lem}{Lemma}[section]
\newtheorem{prop}{Proposition}[section]
\newtheorem{cor}{Corollary}[section]

\newtheorem{ass}{Assumption}[section]
\newtheorem{pro}{Problem}[section]

\begin{document}

\title*{The Relationship between Maximum Principle and Dynamic Programming
Principle for Stochastic Recursive Control Problem with Random Coefficients\thanks{This is work was supported by National Key R\&D Program of China (No.2018YFA0703900), National Natural Science Foundation of China (Nos.11631004, 11871163).}}
\titlerunning{Stochastic Maximum Principle and Dynamic Programming Principle}
\author{Yuchao Dong \inst{1}\and Qingxin, Meng\inst{2} \and Qi, Zhang\inst{3}}
\authorrunning{Y. Dong, Q. Meng and Q. Zhang}
\institute{Institute of Operations Research and Analytics, National University of Singapore, Singapore, dyc19881021@icloud.com\and Department of Mathematics, Huzhou University, Zhejiang, China, mqx@zjhu.edu.cn\and The School of Mathematical Science, Fudan University, Shanghai 200433, China, qzh@fudan.edu.cn}


\maketitle

\begin{abstract}
This paper aims to explore the relationship between maximum principle and dynamic programming
principle for stochastic recursive control problem with random coefficients. Under certain regular conditions for the coefficients, the relationship between the Hamilton system with random coefficients and stochastic Hamilton-Jacobi-Bellman equation is obtained. It is very different from the deterministic coefficients case since stochastic Hamilton-Jacobi-Bellman equation is a backward stochastic partial differential equation with solution being a pair of random fields rather than a deterministic function. A linear quadratic recursive utility optimization problem  is given as an explicitly illustrated example based on this kind of relationship.
\end{abstract}
\section{Introduction}
As we all know, Pontryagin maximum principle (MP) and Bellman dynamic programming principle (DPP) serve as the most two important methods in  solving optimal control problems. Both of them aim to obtain some necessary conditions of optimal controls. Hence it is natural to think that they have some kind of relationship, although they have been developed separately and independently in literature to a great extent. In general, the MP gives a necessity condition of the optimal control by the Hamilton system which is a forward-backward equation consisting of the optimal state equation, the adjoint equation and optimality condition. On the other hand, the DPP characterizes the the optimal control by the Hamilton-Jacobi-Bellman (HJB) equation, to which the value function is a solution. Therefore, the relationship between Hamiltonian system and HJB equation can be thought as a relationship between MP and DPP.

For the deterministic control system, the Hamiltonian system is an ordinary differential equation and the HJB equation is a first-order partial differential equation (PDE), whose connection was first given by Pontryagin, Boltyanski, Gamkrelidze and Mischenko \cite{po-bo-ga-mi} in 1962. Since the value function $V$ is not always smooth, some nonsmooth versions of the relationship were studied by using nonsmooth analysis and generalized derivatives. For example, an attempt to relate these two without assuming the smoothness of the value function was done by Barron and Jensen \cite{Barron1986}, where the viscosity solution was used to derive the MP from the DPP. The relationship in deterministic case is known as 
$$\Psi_t=-V_x(t,\bar X_t)\ \ {\rm and}\ \ V_t(t,\bar X_t)=H(t,\bar X_t,\bar u_t,\Psi_t),$$  where $\bar u$ is the optimal control, $\bar X$ is the optimal state, $\Psi$ is the adjoint variable, $H$ is the Hamiltonian function, and $V$ is the value function, respectively. 
For the stochastic control system whose state equation is a stochastic differential equation (SDE) with deterministic coefficients, the Hamiltonian system is a forward-backward stochastic differential equation (FBSDE) with deterministic coefficients and the HJB equation a second-order fully nonlinear PDE, their connection was given by Bismut \cite{Bismut1978} and Bensoussan \cite{Bensoussan1982}. As for nonsmooth value function, Zhou \cite{Zhou1990, Zhou1991} obtained the relationship between them in the viscosity sense of HJB equation. The relationship in this case can be  summarized as
$$p_t=-V_{x}(t, \bar{x}_t),\ \ \ q_t=-V_{x x}(t, \bar{x}_t) \sigma(t, \bar{x}_t, \bar{u}_t),$$ and
$$V_{t}(t, \bar{x}_t)=G\left(t, \bar{x}_t, \bar{u}_t,-V_{x}(t, \bar{x}_t),-V_{x x}(t, \bar{x}_t)\right),$$ where $\sigma$ is the diffusion coefficient, $(p, q)$ is the adjoint pair and
$G$ is the generalized Hamiltonian function. 

However, when the state equation is a SDE with random coefficients, things are much different. Bear in mind that HJB equation in this case is a backward stochastic partial differential equation (BSPDE) with a pair of adapted solution, rather than a deterministic PDE with a deterministic solution. There should be also a relationship between MP and DPP, as well as between FBSDE with random coefficients and stochastic HJB equation, but no existing literature is concerned with this issue as far as we know.

The relationship between MP and DPP not only demonstrates the connection between two main methods of control theory, but also plays a very important role in economic theory as pointed out in Yong and Zhou \cite{Yong1999}. Moreover, the relationship can be regarded as an extension of Feynman-Kac formula to fully nonlinear PDE, if one notices that the Hamiltonian system is a stochastic forward-backward system and HJB equation is a fully nonlinear PDE in a stochastic control system with deterministic coefficients. For the random coefficients settings,  Feynman-Kac formula is further extended to non-Markovian framework and fully nonlinear BSPDE. The reader can refer to \cite{du-zhang,hu-ma-yo,ma-yo,tang2005semi} for related studies.

The control system we consider to find the relationship between MP and DPP is the stochastic recursive control system with a general cost funtional, which is governed by the following controlled FBSDE:
\begin{equation*}%
\displaystyle\left\{
\begin{array}{lll}
dX_s&=&b(s,X_s,u_s\big)ds+\sigma\big(s,X_s,u_s)dW_s,\\
dY_s&=&-f(s,X_s,Y_s,Z_s, u_s)ds+Z_sdW_s,
\\X_0&=&x,
\\ Y_T&=&h(X_T),
\end{array}
\right.
\end{equation*}
and the following cost functional:
\begin{eqnarray*}
	J(0,x;u(\cdot))\triangleq Y^{0,x;u}_0.
\end{eqnarray*}
The above stochastic recursive control system was given by Peng \cite{Peng1992} to establish DPP in the Lipschitz setting of the generator and explore the connection between its value function and HJB equation. On the other hand,
Duffie and Epstein \cite{Duffie} studied such a control system from mathematical finance point of view, i.e. they put forward the stochastic (recursive) differential utility which can be regarded as the solution of FBSDE.

From MP point of view, Peng \cite{Peng1993} also studied the above recursive control system and derived a local
MP by representing the adjoint
equation as a FBSDE, in which the control domain is convex. 
For the general settings that the control domain is nonconvex and the diffusion depends on control, the Ekeland variational principle was applied to obtain the MP in Wu \cite{Wu2013} and Yong \cite{Yong10} by treating the second solution
and the terminal condition in backward stochastic differential equation (BSDE) as a control and a constraint, respectively. By introducing new and general first-order and second-order adjoint equations, Hu \cite{Hu2017} obtained the MP for the recursive stochastic optimal control problem without unknown parameters. These results, especially Hu \cite{Hu2017}, eventually solved the long-standing open problem put forward in Peng \cite{peng1998}.


There has been results on the relationship between MP and DPP for stochastic recursive
optimal control system with deterministic coefficients. With sufficiently regular assumptions on the coefficients, Shi \cite{Shi2010} and Shi and Wu \cite{Shi2011} first demonstrated this relationship. Nie, Shi and Wu \cite{Nie2016,Nie2017} studied the relationship between MP and DPP in the sense of viscosity solution of HJB equation. The relationship is summarized as follows:
$$
\begin{array}{l}
\left\{\begin{aligned}
p^{*}_t=& V_{x}(t, \bar{x}_t)^{\top} q^{*}_t, \\
k^{*}_t=&\left[V_{x x}(t, \bar{x}_t) \sigma(t, \bar{x}_t, \bar{u}_t)+V_{x}(t, \bar{x}_t)\right.\\
&\left.\times f_{z}\left(t, \bar{x}_t,-V(t, \bar{x}_t),-V_{x}(t, \bar{x}_t) \sigma(t, \bar{x}_t, \bar{u}_t), \bar{u}_t\right)\right] q^{*}_t
\end{aligned}\right.
\end{array}
$$
and
$$V_{t}(t, \bar{x}_t)=G\left(t, \bar{x}_t,-V\left(t, \bar{x}_t,-V_{x}(t, \bar{x}_t),-V_{x x}(t, \bar{x}_t), \bar{u}_t\right),\right.$$
where $(p^*,q^*)$ is the adjoint pair of the forward part, $k^*$ is the adjoint process of the backward part in stochastic recursive control system and $G$ is the corresponding generalized Hamiltonian function.

In our paper, the most important feature is that the coefficients of the system we consider are random. We emphasize that this is an essential difference from existing literature. In 1992, Peng [43] studied the optimal control problem of non-Markovian stochastic systems using dynamic programming. Compared with the optimal control problem of Markov stochastic systems, the value function is no longer a deterministic function, but a random field. In other words, it is a family of semi-martingales. Furthermore, the HJB equation derived from Bellman's principle of optimality is no longer a second-order fully nonlinear PDE, but a second-order fully nonlinear BSPDE, whose solution is a pair of random fields as BSDE's. To distinguish it from the classical HJB equation, we call it the stochastic HJB equation.  As in the deterministic case, the existence of the solution for stochastic HJB equation is a very hard problem. The solvability has only been proved for a few cases, see \cite{tang2003general,tang2015dynamic,qiu2017,qiu2018,zhang2020backward} for instance. One contribution of our paper is to show that the value function of the recursive optimal control problem will be the classical solution of stochastic HJB equation, if the needed regularity is satisfied. It can be seen as a general form of Feyman-Kac representation. In this sense, our work extends the result of Tang \cite{tang2005semi}, in which the author used a forward-backward system to represent  semilinear backward stochastic partial  differential equation. In fact, our proof is partly inspired from that work, i.e. we also use the random field generated by the controlled SDE. Furthermore, we proved a verification theorem to show that the solution of stochastic HJB equation gives the optimal control.  Another contribution of our paper is to show the connection between the MP and the DPP. Our result extends those for stochastic recursive
optimal control system with deterministic coefficients. Note that we also assume that the value function is smooth to obtain the desired result, but how to deal with nonsmooth case is still unsolved. Actually, the solvability for the stochastic HJB equation in a general form is a long-existing open problem.

The rest of this article is organized as follows. In Section 2, we introduce some notations and the basic setup of our problem. We characterize the optimal control by DPP, i.e. the relation between the value function and stochastic HJB equation in Section 3. In Section 4, the optimal control is characterized by MP, i.e. the stochastic Hamiltonian system. In Section 5, we show the connection between the MP and the DPP.
As an application
we discuss a
linear quadratic (LQ) recursive utility portfolio optimization
problem with the random coefficients in Section 6, in which the
state feedback optimal control is obtained by both MP and DPP methods,
and the relations we obtained are demonstrated explicitly.

\section{Notations \& Statement of the problem}
Let $(\Omega,
\mathscr{F}, P)$ be a complete
probability space, and $\{W_t, 0\leq t\leq T\}$ is a one-dimensional standard
Brownian motion on it generating a right-continuous filtration $\{\mathscr{F}_t\}_{0\leq t\leq
T}$. Let $E$ be an Euclidean
space, and its inner product and norm are denoted by $(\cdot, \cdot)$ and
$|\cdot|$, respectively. For a function
$\phi:\mathbb R^n\longrightarrow \mathbb R$, we denote by $\phi_x$ its
gradient and by $\phi_{xx}$ its Hessian (a symmetric matrix). If $\phi:
\mathbb R^n\longrightarrow \mathbb R^k$ ($k\geq 2$),
$\phi_x=(\frac{\partial \phi_i}{\partial x_j})$ is the corresponding
$k\times n$ Jacobian matrix. 

Next we introduce some useful spaces of random variables and stochastic
processes. For any $\alpha\in [1,\infty)$ and $\beta \in£¨(0,\infty)$, we let:

$\bullet$~~$M_{\mathscr{F}}^\beta(0,T;E)$: the space of all ${%
\mathscr{F}}_t$-adapted processes $f:\Omega\times[0,T]\rightarrow E$ satisfying $
\|f\|_{M_{\mathscr{F}}^\beta(0,T;E)}\triangleq{\left (\mathbb E\displaystyle%
\int_0^T|f_t|^ \beta dt\right)^{1\wedge\frac{1}{\beta}}}<\infty. $

$\bullet$~~$S_{\mathscr{F}}^\beta (0,T;E)$: the space of all ${%
\mathscr{F}}_t$-adapted c\`{a}dl\`{a}g processes $f:\Omega\times[0,T]\rightarrow E$ satisfying $
\|f\|_{S_{\mathscr{F}}^\beta(0,T;E)}\triangleq{\left (\mathbb
E\displaystyle\sup_{t\in [0,T]}|f_t|^\beta dt\right)^{1\wedge\frac
{1}{\beta}}}<+\infty. $

$\bullet$~~$L^\beta (\Omega;E)$: the space of all
random variables $\xi:\Omega\rightarrow E$
satisfying $ \|\xi\|_{L^\beta(\Omega;E)}\triangleq
\left(\mathbb E|\xi|^\beta\right)^{1\wedge\frac{1}{\beta}}<\infty$.

$\bullet$~~${\color{black}M_{\mathscr{F}}^\beta(L^\alpha ([0,T]; E))}$: the space of all ${%
\mathscr{F}}_t$-adapted processes $f:\Omega\times[0,T]\rightarrow E$
satisfying $
\|f\|_{\alpha,\beta}\triangleq{\left[\mathbb E\left(\displaystyle%
\int_0^T|f_t|^\alpha
dt\right)^{\frac{\beta}{\alpha}}\right]^{1\wedge\frac{1}{\beta}}}<\infty. $

%

For any $t,s \in [0,T]$ with  $t\leq s$, we define the admissible control set ${\cal
U}^{2}[t,s]=M_{\mathscr{F}}^{2}(L^{2}([t,s]; U))$ with $U$ being  a closed convex subset of $\mathbb R^k$. 
Given  $x\in\mathbb R^n$ and  $u\in {\cal U}^2[t,T]$, we consider the
following FBSDE
\begin{equation}\label{mz19}%
\displaystyle\left\{
\begin{array}{lll}
dX^{0,x;u}_s&=&b(s,X^{0,x;u}_s,u_s\big)ds+\sigma\big(s,X^{0,x;u}_s,u_s)dW_s,\\
dY^{0,x;u}_s&=&-f(s,X^{0,x;u}_s,Y^{0,x;u}_s,Z^{0,x;u}_s, u_s)ds+Z^{0,x;u}_sdW_s,
\\X^{0,x;u}_0&=&x,
\\ Y^{0,x;u}_T&=&h(X^{0,x;u}_T),
\end{array}
\right.
\end{equation}
with the cost functional
\begin{eqnarray*}
J(0,x;u)\triangleq Y^{0,x;u}_0,
\end{eqnarray*}
where $b: \Omega\times[0,T]\times \mathbb{R}^n\times U\rightarrow
\mathbb{R}^n,\ \sigma:\Omega\times[0,T]\times \mathbb{R}^n\times
U\rightarrow \mathbb{R}^{n},\ f:\Omega\times[0,T]\times
\mathbb{R}^n\times \mathbb{R}\times \mathbb{R}\times
U\longrightarrow \mathbb{R},\ h: \Omega\times
\mathbb{R}^n\longrightarrow \mathbb{R}$. 

We need the following assumptions on coefficients $(b,\sigma, f, h)$.
\begin{ass}\label{assum_b}
	For any $(\omega,t,x, u)\in\Omega\times[0,T]\times\mathbb R^n\times U$, $ b(\cdot,x,u)$ and
	$\sigma(\cdot,x, u)$ are ${\mathscr F}_t$-adapted processes; $b(t,\cdot,u)$, $\sigma(t,\cdot,u) \in C^2(\mathbb R^n,\mathbb R^n)$; $b_{x}(t,x,u)$, $\sigma_{x}(t,x,u)$, $b_{u}(t,x,u)$, $\sigma_{u}(t,x,u)$ are continuous in $(x,u)$; there exists a constant $K$ such that
	$$|b(t,x,u)|,|\sigma(t,x,u)|\le K£¨(1+|x|+|u|)\ \ and\ \ |b_x|,|b_u|,|b_{xx}|,|\sigma_x|,|\sigma_u|,|\sigma_{xx}|\le K.$$
\end{ass}
\begin{ass}\label{assum_f}
For any $(\omega, t,x,x_1,x_2,y,z, u,u_1,u_2)\in\Omega\times[0,T]\times\mathbb R^n\times\mathbb R^n\times\mathbb R^n\times \mathbb R\times \mathbb R\times U\times U\times U$, $f(\cdot,x,y,z,u)$ is an ${\mathscr F}_t$-adapted process and $h(x)$ an $
\mathscr{F}_{T}$-measurable random variable; $f$ is
differentiable with respect to $(x,y,z,u)$ and  $h$ is
differentiable with respect to $x$; $f_x(t,x,y,z,u)$, $f_y(t,x,y,z,u)$, $f_z(t,x,y,z,u)$, $f_u(t,x,y,z,u)$ are continuous in $(x,y,z,u)$, $h_x(x)$ is continuous in $x$; there exists a constant $K$ such that for $\gamma \in [0,1)$
$$
|f(t,x,y,z,u)| \le K(1+|x|^2+|y|+|z|^{\gamma}+|u|^2),\ \ |h(x)|\le K(1+|x|^2),
$$
$$
|f_y|,|f_z| \le K,\ \ |f_x(t,x,y,z,u)| \le K(1+|x|+|u|)
$$
and
\begin{eqnarray*}
&&|h(x_1)-h(x_2)|+|f(t,x_1, y,z,u_1)-f(t,x_2, y,z, u_2)|\\
&\leq&K(1+|x_1|+|x_2|)(|x_1-x_2|)+K(1+|u_1|+|u_2|)|u_1-u_2|.
\end{eqnarray*}
\end{ass}
Under Assumption \ref{assum_b}, we can see that, for any given admissible control $u$,
the forward part of SDE (\ref{mz19}) admits a unique strong solution $X^u\in S_{\mathscr{F}}^2 (0,T;\mathbb R^n)$. Thus, we see that the terminal $h(X^u_T)$ is only $L^1$-integrable. Thanks to the sublinear growth of $f$ with respect to $z$ and Theorem 6.3 in \cite{briand2003lp}, there exists a unique solution $(Y^{u},Z^{u}) \in S_{\mathscr{F}}^\beta (0,T;\mathbb R)\times M_{\mathscr{F}}^\beta (0,T;\mathbb R)$  for any $\beta\in(0,1)$. 
It is easy to check that $| J ( 0,x; u  ) | < \infty$.
Then, we put forward the optimal control problem.

\begin{pro}
\label{pro:2.1} Find an admissible control $\bar{u}$ such
that
\begin{equation}  \label{eq:b7}
J(0,x;\bar{u})=\displaystyle\inf_{u\in {\cal
U}^{2}[t,T]}J(0,x;u).
\end{equation}
\end{pro}
Any $\bar{u}\in {\cal U}^{2}[0,T]$ satisfying \eqref{eq:b7} is
called an optimal control process of Problem \ref{pro:2.1}. With $\bar{u}$, the solution  $(\bar{X},\bar{Y}, \bar
{Z})$ of the state equation (\ref{mz19}) is called the optimal
state process, and consequently $(\bar{u};\bar{X}, \bar{Y}, \bar {Z})$ is called an optimal pair of Problem \ref{pro:2.1}.

\section{ The Dynamic Programming Principle and Stochastic HJB Equation for Stochastic Recursive Control Problem}
In this section, we are concerned with the dynamic programming principle and the corresponding
stochastic HJB Equation for stochastic recursive control Problem \ref{pro:2.1}. We shall show that, if the value function is a random field with some regularities, it will be the solution for the stochastic HJB equation.
To this end,
for $t\in
[0,T]$ and $\zeta\in L^2 (\Omega;\mathbb R^n)$ and $u\in {\cal U}^2[t,T]$, we consider the following parameterized FBSDE:
\begin{equation}\label{eq:3.1}%
\displaystyle\left\{
\begin{array}{lll}
dX^{\zeta,x;u}_s&=&b(s,X^{t,\zeta,;u}_s,u_s\big)ds+\sigma\big(t,X^{t,\zeta;u}_s,u_s)dW_s,\\
dY^{\zeta,x;u}_s&=&-f(s,X^{t,\zeta;u}_s,Y^{t,\zeta;u}_s,Z^{t,\zeta;u}_s, u_s)ds+Z^{t,\zeta;u}_sdW_s,
\\X^{t,\zeta;u}_t&=&\zeta,
\\ Y^{t,\zeta;u}_T&=&h(X^{t,\zeta;u}_T).
\end{array}
\right.
\end{equation}
{\color{black}Under Assumption \ref{assum_b} and \ref{assum_f},  by Theorem 6.3 in \cite{briand2003lp} again,
FBSDE (\ref{eq:3.1}) admits a unique strong solution $\Theta^{t,\zeta;u}=(X^{t,\zeta;u},Y^{t,\zeta;u},Z^{t,\zeta;u})\in S_{\mathscr{F}}^2 (t,T;\mathbb R^n)\times
S_{\mathscr{F}}^\beta (t,T;\mathbb R)\times M_{\mathscr{F}}^\beta(t,T; \mathbb R)$ for any $\beta \in(0,1)$. We call $\Theta^{t,\zeta;u}$, or $\Theta=(X,Y,Z)$
whenever its dependence on $u$ and $(t,\zeta)$ is clear from context,
the state process and $(u;\Theta)$
is the admissible pair}

For a given control process $u\in {\cal U}^2[t,T]$, we define the associated cost functional as follows.
\begin{eqnarray*}
J(t,x;u)\triangleq Y_{t}^{t,x;u},\ \ \  (t,x)\in [0,T]\times \mathbb R^n.
\end{eqnarray*}
From  Theorem A.2 in \cite{Peng1997}, we get the following relation
\begin{eqnarray}
    J(t,\zeta,u)=Y_{t}^{t,\zeta;u}.
\end{eqnarray}
For $\zeta=x\in \mathbb R^n, $ the value function we define in this part is
\begin{eqnarray*}
V(t,x)\triangleq\mathop{\text{essinf}}\limits_{u\in
{\cal U}^2[t,T]}J(t,x;u),\ \ \  (t,x)\in [0,T]\times \mathbb R^n.
\end{eqnarray*}

Now we discuss a generalized DPP for our stochastic optimal
control problem. For this purpose, 
we define the family of (backward) semigroups associated with FBSDE (\ref{eq:3.1}), which was first introduced by Peng  \cite{Peng1997}. Given the initial data $(t, x)$, a positive number $\delta \leq
T-t,$ an admissible control process $u \in {\cal U}^2[t, t+\delta]$
and a real-valued random variable $\eta \in L^2
(\Omega,\mathscr F_{t+\delta};\mathbb R),$ we put
\begin{equation}
  G_{s,t+\delta }^{t,x;u}(\eta): =\tilde Y_{s}^{t,x;u},\ \ \ s\in[t,
t+\delta]
\end{equation}
where  $(X^{t,x,u}_\cdot,\tilde Y^{t,x;u}_\cdot, \tilde
Z^{t,x;u}_\cdot)$ is the solution of the following FBSDE with the time
horizon $t+\delta,$

\begin{equation}\label{eq:3.2}%
\displaystyle\left\{
\begin{array}{lll}
dX^{t,x;u} _s&=&b(s,X^{t,x;u}_s,u_s\big)ds+\sigma\big(s,X^{t,x;u}_s,u_s)dW_s,\\
d\tilde Y^{t,x;u} _s&=&-f(s,X^{t,x;u}_s,\tilde Y^{t,x;u}_s,\tilde Z^{t,x;u}_s, u_s)ds+\tilde Z^{t,x;u}_sdW_s,
\\X^{t,x;u}_t&=&x,
\\Y^{t,x;u}_{t+\delta}&=&\eta.
\end{array}
\right.
\end{equation}
Obviously, for any admissible control pair
$(X^{t,x,u}, Y^{t,x;u},
Z^{t,x;u};u),$ we have
\begin{eqnarray}
G_{t,T}^{t,x;u}\big(h(X^{t,x;u}_T)\big) =G_{t,t+\delta }^{t,x;u}(
Y_{t+\delta}^{t,x;u}) =G_{t,t+\delta }^{t,x;u}(Y_{t+\delta
}^{t+\delta,X_{t+\delta }^{t,x;u};u}) =G_{t,t+\delta }^{t,x;u}(
J(t+\delta,X_{t+\delta }^{t,x;u};u)).
\end{eqnarray}
Moreover, the following dynamic programming principle holds by a similar proof as in \cite{Peng1997}.
\begin{thm}
  Under Assumption \ref{assum_b} and \ref{assum_f}, the value function $v(t,x)$
  obeys the following DPP: for any $0 \leq t<t+\delta\leq T,x\in \mathbb{R}^n$,
  \begin{eqnarray*}
V(t,x)=\inf\limits_{u\in
\mathcal{U}^2[t,t+\delta]}G_{t,t+\delta}^{t,x;u}\big(V(t+\delta ,X_{t+\delta
}^{t,x;u})\big).
\end{eqnarray*}
\end{thm}
Next we shall show the relation between the value function and stochastic HJB equation.  For this purpose, the following lemma in \cite{tang2005semi} is needed.
\begin{lem}\label{lem_flow}
For any fixed admissible control $u$, set $\mathbb X_s^x$ to be the solution of the following SDE:
\begin{equation}\label{DMZ3}
\displaystyle\left\{
\begin{array}{lll}
dX_s=b(s, X_s,u_s)ds+\sigma(s,X_s,u_s)dW_s,\\
X_0=x.
\end{array}
\right.
\end{equation}
Then, almost surely, for each $s \in [0,T]$, $\mathbb X_s^\cdot$ is a diffeomorphism of $C^1$. The gradient $\partial\mathbb X_s^x$ satisfies the following SDE:
\begin{equation*}
\displaystyle\left\{
\begin{array}{lll}
d\partial \mathbb X_s^x =b_x(s,\mathbb X_s^x,u_s)\partial \mathbb X_s^x ds+\sigma_x(s,\mathbb X_s^x,u_s)\partial \mathbb X_s^x  dW_s,\\
\partial \mathbb X_0^x =I.
\end{array}
\right.
\end{equation*}
Moreover, from the boundedness of the derivatives, classical estimation for SDE yields that
$$
\mathbb E\left[ \sup_{s \in [0,T]} |\partial \mathbb X_s^x|^4 \right] \le M,
$$
where  $M$ is  a constant independent of $x$.
\end{lem}
Then main result of this section is presented below.
\begin{prop}\label{prop_verfication}
In additional to Assumptions \ref{assum_b} and \ref{assum_f}, we also assume that the control region {\color{black}$U\subset\mathbb{R}^k$ is bounded and, for each $t\in[0,T]$ and $x\in\mathbb{R}^n$}, the infimum of the cost functional $J(t,x;\cdot)$ is attained by an optimal control $u^{*,t,x}$. 
 Moreover,  assume that the value function $V(t,x)$ admits the following
 semimartingale decomposition:
\begin{eqnarray}\label{mz12}
V(t,x) =h(x)+\int_{t}^{T}\Gamma(s,x)ds-\int_{t}^{T}\Psi(s,x)dW_{s},\
\ \ t\in[0,T],
\end{eqnarray}
where  the  $\mathbb{R}$-valued function $\Gamma(t,\cdot)$ and
$\Psi(t,\cdot)$ are  $\mathscr F_t\times \mathcal B(\mathbb R^n)$ measurable for each $t\in[0,T]$ and
 $V,\Gamma,\Psi$ satisfy the following assumptions:
\begin{enumerate}[label=(\roman*)]
	\item $(t,x)\longmapsto V(t,x)  $ is continuous a.s.,
	\item $x\longmapsto V(t,x)$ is $C^2$ for each $t\in [0,T]$ a.s.,
	\item $x\longmapsto \Gamma(t,x)$ is continuous for each $t\in
	[0,T]$ a.s.,
	\item $x\longmapsto \Psi(t,x)$ is  is $C^1$  for each $t\in
	[0,T]$ a.s.,
	\item There exists $K \in {\color{black}M_{\mathscr{F}}^2(L^2(0,T;\mathbb R^+))}$ such that
	$$
	|V(t,x)|,\ |h(t,x)|,\ |\Gamma(t,x)|,\ |\Psi(t,x)| \le K_t(1+|x|^2),
	$$
	$$
	|\partial_x V(t,x)|,\ |\partial_x \Psi(t,x)| \le 	K_t(1+|x|),$$
	$$
	|\partial_{xx} V(t,x)| \le K_t,
	$$
	$$
	|\Gamma(t,x)-\Gamma(t,y)|\le K_t(1+|x|+|y|)|x-y|.
	$$
\end{enumerate}
Then, the value function $V$, together with $\Psi$,
constitutes a pair solution of the so-called backward HJB equation
\begin{numcases}{}\label{mz4}
dV(t,x)=-\inf_{u}G\big(t,x,V(t,x),\Psi(t,x),V_{x}(t,x),\Psi
_{x}(t,x),V_{xx}(t,x),u\big)dt+\Psi(t,x) dW_{t},\nonumber\\
V(T,x)=h(x),
\end{numcases}
where
\begin{eqnarray*}
	G(t,x,y,z,p,q,A,u)&=&\langle p,b(t,x,u)\rangle+\langle q,\sigma(t,x,u)\rangle +{1\over2}tr\big((\sigma\sigma^{*})(t,x,u)A\big)  \\
	&&+f(t,x,y,\sigma^*p+z,u).
\end{eqnarray*}
\end{prop}
\begin{proof}
Let $\{ x_i\}=\mathbb Q$. For a fixed  admissible control $u$ and $x_i$,  we abbreviate $X$ for $X^{0,x_i;u}$ for simplicity. Applying It\^o-Ventzell formula to   $V(t,X_t)$, we have
\begin{eqnarray*}
	V(t,X_t)=V(t+\delta,X_{t+\delta})+\int_t^{t+\delta}&& \Gamma(s,X_s)-G(s,X_s,V(s,X_s),\Psi(s,X_s),V_x(s,X_s),\Psi_x(s,X_s),V_{xx}(s,X_s),u_s)\\
	&&+f(s,X_s,V(s,X_s),Z'_s,u_s)ds-\int_t^{t+\delta}Z'_s dW_s,
\end{eqnarray*}
where $Z'_s=\sigma^*V_x(s,X_s)+\Psi(s,X_s)$. From condition (v) in the theorem, it can be verified that
$Z' \in M_{\mathscr{F}}^{2}(0,T)$. Consider the following BSDE
\begin{eqnarray*}
	Y_r=V(t+\delta,X_{t+\delta})+\int_r^{t+\delta} f(s,X_s,Y_s,Z_s,u_s)ds-\int_r^{t+\delta}Z_sdW_s.
\end{eqnarray*}
From the DPP, we shall have that $Y_t \ge V(t,X_t)$. After linearization, $Y_t-V(t,X_t)$ can be written as
\begin{eqnarray}\label{DMZ5}
	V(t,X_t)-Y_t=\mathbb E\left[ \int_t^{t+\delta} \xi_s\Delta(s,X_s,u_s)ds \bigg |\mathcal F_t\right]\le 0,
\end{eqnarray}
where
$$
\Delta(s,x,u)=\Gamma(s,x)-G(s,x,V(s,x),\Psi(s,x),V_x(s,x),\Psi_x(s,x),V_{xx}(s,x),u)
$$
and $\xi_s$ satisfies the following SDE:
\begin{equation*}
\displaystyle\left\{
\begin{array}{lll}
d\xi_s=A_s \xi_s ds+B_s \xi_sdW_s,\\
\xi_t=1
\end{array}
\right.
\end{equation*}
with
$$
A_s=\frac{f(s,X_s,V(s,X_s),Z'_s,u_s)-f(s,X_s,Y_s,Z'_s,u_s)}{V(s,X_s)-Y_s},
$$
and
$$
B_s=\frac{f(s,X_s,Y_s,Z'_s,u_s)-f(s,X_s,Y_s,Z_s,u_s)}{Z'_s-Z_s}.
$$
Since $f(t,x,y,z,u)$ is Lipschitz continuous with respect to $y$ and $z$, it is easy to see that $A$ and $B$  are uniformly bounded processes. Then, the classical estimation for linear SDEs yields that
\begin{eqnarray}\label{DMZ1}
\mathbb E\left[ |\xi_s-1|^2 |\mathcal F_t \right] \le C\mathbb E \left[  \left(\int_t^s |A_s|ds\right)^2 +\int_t^s |B_s|^2ds  \bigg|\mathcal F_t \right] \le  C(|t-s|+|t-s|^2).
\end{eqnarray}
Here and throughout this paper, $C$ is a generic constant whose values may change from line by line.
To emphasize its dependence on $t$ and $\delta$, we also denote $\xi$ as $\xi^{t,\delta}$. Then, we claim that, for any $t$ and $\delta$,
\begin{eqnarray}\label{ineq_Q_expec}
	\mathbb E\left[ \int_t^{t+\delta}\Delta(s,X_s,u_s)ds \bigg|\mathcal F_t\right]\le 0,\ \ \ a.s.
\end{eqnarray}
To see this, for fixed $t$ and $\delta$,  we obtain similarly that, for any $n$ and $k\le n$,
\begin{equation}\label{esti_linear_SDE}
\mathbb E\left[ \int_{t+\frac{k}{n}\delta}^{t+\frac{k+1}{n}\delta} \xi^{t+\frac{k}{n}\delta,\frac{\delta}{n}}_s\Delta(s,X_s,u_s)ds\bigg |\mathcal F_t\right]\le 0.
\end{equation}
Then, from \eqref{esti_linear_SDE}, we have
\begin{equation*}
\begin{split}
&\mathbb E\left[  \int_{t+\frac{k}{n}\delta}^{t+\frac{k+1}{n}\delta}\Delta(s,X_s,u_s)ds\bigg |\mathcal F_t\right]\\
=& \mathbb E\left[ \int_{t+\frac{k}{n}\delta}^{t+\frac{k+1}{n}\delta} \xi ^{t+\frac{k}{n}\delta,\frac{\delta}{n}}_s\Delta(s,X_s,u_s)ds\bigg |\mathcal F_t\right]+\mathbb E\left[ \int_{t+\frac{k}{n}\delta}^{t+\frac{k+1}{n}\delta}(1- \xi ^{t+\frac{k}{n}\delta,\frac{\delta}{n}}_s)\Delta(s,X_s,u_s)ds\bigg| \mathcal F_t\right]\\
\le &\mathbb E\left[ \int_{t+\frac{k}{n}\delta}^{t+\frac{k+1}{n}\delta}(1-\xi ^{t+\frac{k}{n}\delta,\frac{\delta}{n}}_s)\Delta(s,X_s,u_s)ds\bigg|\mathcal F_t\right]\\
\le & \left( \mathbb E\left[  \int_{t+\frac{k}{n}\delta}^{t+\frac{k+1}{n}\delta}(1-\xi ^{t+\frac{k}{n}\delta,\frac{\delta}{n}}_s)^2ds  \bigg|\mathcal F_t\right] \right)^{1/2}\left( \mathbb E\left[   \int_{t+\frac{k}{n}\delta}^{t+\frac{k+1}{n}\delta}| \Delta(s,X_s,u_s)|^2ds\bigg|\mathcal F_t  \right] \right)^{1/2}
\end{split}
\end{equation*}
Summing over $k$, we have
\begin{equation}\label{DMZ2}
\begin{split}
&\mathbb E\left[  \int_{t}^{t+\delta}\Delta(s,X_s,u_s)ds\bigg |\mathcal F_t\right]\\
\le& \sum_{k=0}^{n-1} \left( \mathbb E\left[  \int_{t+\frac{k}{n}\delta}^{t+\frac{k+1}{n}\delta}(1-\xi ^{t+\frac{k}{n}\delta,\frac{\delta}{n}}_s)^2ds  \bigg|\mathcal F_t\right] \right)^{1/2}\left( \mathbb E\left[   \int_{t+\frac{k}{n}\delta}^{t+\frac{k+1}{n}\delta}| \Delta(s,X_s,u_s)|^2ds\bigg|\mathcal F_t  \right] \right)^{1/2}\\
\le &\left( \sum_{k=0}^{n-1} \mathbb E\left[  \int_{t+\frac{k}{n}\delta}^{t+\frac{k+1}{n}\delta}(1-\xi ^{t+\frac{k}{n}\delta,\frac{\delta}{n}}_s)^2ds  \bigg|\mathcal F_t\right] \right)^{1/2} {\color{black}\left( \mathbb E\left[   \int_{t}^{t+\delta}| \Delta(s,X_s,u_s)|^2ds\bigg|\mathcal F_t  \right] \right)^{1/2}},
\end{split}
\end{equation}
where the last inequality is obtained due to H\"older inequality. By \eqref{DMZ1}, we have
\begin{equation*}
\begin{split}
&\mathbb E\left[  \int_{t+\frac{k}{n}\delta}^{t+\frac{k+1}{n}\delta}(1-\xi ^{t+\frac{k}{n}\delta,\frac{\delta}{n}}_s)^2ds \bigg |\mathcal F_t\right]\\
=&\int_{t+\frac{k}{n}\delta}^{t+\frac{k+1}{n}\delta} \mathbb E\left[  (1-\xi ^{t+\frac{k}{n}\delta,\frac{\delta}{n}}_s)^2 \bigg |\mathcal F_t\right]ds\\
\le& C\int_{t+\frac{k}{n}\delta}^{t+\frac{k+1}{n}\delta} |t+\frac{k}{n}\delta-s|+|t+\frac{k}{n}\delta-s|^2 ds
\\
\le & C\frac{\delta^2}{n^2}.
\end{split}
\end{equation*}
Thus,
$$
\left( \sum_{k=0}^{n-1} \mathbb E\left[  \int_{t+\frac{k}{n}\delta}^{t+\frac{k+1}{n}\delta}(1-\xi ^{t+\frac{k}{n}\delta,\frac{\delta}{n}}_s)^2ds  \bigg|\mathcal F_t\right] \right)^{1/2} \longrightarrow 0,\ \ \  \text{as $n \rightarrow \infty$.}
$$
{\color{black}Due to Assumption \ref{assum_b} and boundedness of the control region $U$, $X\in S_{\mathscr{F}}^\beta(0,T;\mathbb{R}^n)$ for any $\beta\geq2$, and thus $\left( \mathbb E\left[   \int_{t}^{t+\delta}| \Delta(s,X_s,u_s)|^2ds\bigg|\mathcal F_t  \right] \right)^{1/2}$ is bounded. Hence, letting $n \rightarrow \infty$ in \eqref{DMZ2}, we have that
\begin{eqnarray}\label{DMZ4}
\mathbb E\left[  \int_{t}^{t+\delta}\Delta(s,X_s,u_s)ds \bigg|\mathcal F_t\right]\leq0.
\end{eqnarray}
For fixed $t\in[0,T]$ and any nonnegative $\mathbb{R}$-valued random variable $\eta\in\mathscr{F}_t$, it follows from (\ref{DMZ4}) that
\begin{eqnarray*}
\mathbb E\left[ \int_{0}^{T}\Delta(s,X_s,u_s)\eta I_{[t,t+\delta)}(s)ds \right]=\mathbb E\left[\eta\mathbb E\left[\int_{t}^{t+\delta}\Delta(s,X_s,u_s)ds \bigg|\mathcal F_t\right]\right]\leq0.
\end{eqnarray*}
Consequently, for any nonnegative simple progress $\phi\in M_{\mathscr{F}}^2(0,T;\mathbb{R})$,
\begin{eqnarray*}
\mathbb E\left[ \int_{0}^{T}\Delta(s,X_s,u_s)\phi_sds \right]\leq0.
\end{eqnarray*}
For any nonnegative progress $\psi\in M_{\mathscr{F}}^2(0,T;\mathbb{R})$, there exists a sequence of nonnegative simple progresses $\phi^n\in M_{\mathscr{F}}^2(0,T;\mathbb{R})$, $n\in\mathbb{N}$, such that
\begin{eqnarray*}
\lim_{n\to\infty}\mathbb E\left[ \int_{0}^{T}|\phi^n_s-\psi_s|^2ds \right]=0.
\end{eqnarray*}
Hence
\begin{eqnarray*}
&&\lim_{n\to\infty}\left|\mathbb E\left[ \int_{0}^{T}\Delta(s,X_s,u_s)\phi^n_sds \right]-\mathbb E\left[ \int_{0}^{T}\Delta(s,X_s,u_s)\psi_sds \right]\right|\\
&\leq&\lim_{n\to\infty}\left( \mathbb E\left[   \int_{0}^{T}|\Delta(s,X_s,u_s)|^2ds\right] \right)^{1/2}\left( \mathbb E\left[   \int_{0}^{T}|\phi^n_s-\psi_s|^2ds\right] \right)^{1/2}=0,
\end{eqnarray*}
which implies that
\begin{eqnarray*}
\mathbb E\left[ \int_{0}^{T}\Delta(s,X_s,u_s)\psi_sds \right]\leq0.
\end{eqnarray*}
Noticing the arbitrariness of nonnegative process $\psi$, we have that
\begin{eqnarray*}
{\color{black}\Delta(s,X_s,u_s) \le 0\ \ \ {\rm for\ a.e.}\ s\in[0,T],\ {\rm a.s.}}
\end{eqnarray*}}
Let $\mathbb X_s^x$ be the stochastic flow generated by the SDE \eqref{lem_flow}.
From Lemma \ref{lem_flow}, with probability $1$, for each $s$, $\mathbb  X_s^\cdot$ is a diffeomorphism of class $C^1$. For each $x_i$, we also have that
$$
\Delta(s,\mathbb X_s^{x_i},u_s) \le 0\ \ \ {\rm for\ a.e.}\ s\in[0,T],\ {\rm a.s.}
$$
Since $\Delta(s,x)$ and ${\mathbb  X_s^{x}}$ is continuous with respect to $x$, we shall get that
$$
\Delta(s,\mathbb X_s^{x},u_s) \le 0\ \ \ {\rm for\ all}\ x\in\mathbb{R}^n,\ {\rm a.e.}\ s\in[0,T],\ {\rm a.s.}
$$
From the growth condition of the coefficients and the value function, we see that
$$
|\Delta(t,\mathbb X_t^{x},u_t)|^2 \le C(1+K^2_t)(1+|\mathbb X_t^x|^4).
$$
Then,
\begin{equation*}
\begin{split}
&\mathbb E\left[  \int_0^T |\Delta(t,\mathbb X_t^{x},u_t)|^2dt\right]\\
\le &C\mathbb E\left[  \int_0^T (1+K_t^2) (1+|\mathbb X_t^{x}|^4)dt\right]\\
\le & C\mathbb E\left[ \sup_t (1+|\mathbb X_t^{x}|^4) \int_0^T (1+K_t^2)dt\right]\\
\le & C\left( E\left[\left( \sup_t(1+|\mathbb X_t^{x}|^4)\right)^2 \right] \right)^{1/2}\left( E\left[\left(  \int_0^T (1+K_t^2)dt\right)^2\right]  \right)^{1/2}\\
\le& C(1+|x|^4).
\end{split}
\end{equation*}

Now, let $\varphi$ be a smooth function such that
$$
\varphi(x)=\left\{
\begin{split}
&1,\text{ for $|x|\le 1$;}\\
&0,\text{ for $|x| \ge 2$;}\\
&\in [0,1],\text{ otherwise}.
\end{split}
\right.
$$
\begin{color}{black}
For $s\in[0,T]$, define $\tilde{\mathbb X}_s^\cdot$ to be the inverse function of ${\mathbb X}_s^\cdot$ and consider a random function
$$
g(s,x)=\xi(\mathbb X_s^{x})\varphi(\frac{x}{N})|\det \partial\tilde{ \mathbb X}_s^y|_{y={\mathbb X}^x_s}|^{-1}p_s,
$$
where $N\in\mathbb{N}$, $p$ is an arbitrarily given bounded non-negative adapted process and $\xi$ is a smooth non-negative function with a compact support. Let us first prove
$\mathbb E\left[ \int_0^T \int_{\mathbb{R}^n} \Delta (s,\mathbb X_s^{x})g(s,x)dxds  \right]<\infty$. By H\"older inequality, it holds that
\begin{equation*}
\begin{split}
\mathbb E\left[ \int_0^T \int_{\mathbb{R}^n}|\Delta (s,\mathbb X_s^{x},u_s)g(s,x)|dxds  \right]\le &\left( \mathbb E\left[ \int_0^T \int_{\mathbb{R}^n}|\Delta(s,\mathbb X_s^{x},u_s) |^2 \varphi(\frac{x}{N})dxds  \right]  \right)^{1/2}\\
&\left( \mathbb E\left[ \int_0^T \int_{\mathbb{R}^n}\xi^2(\mathbb X_s^{x})\varphi(\frac{x}{N})|\det \partial_y \tilde{\mathbb X}_s^y|_{y= {\mathbb X}^x_s}|^{-2}p^2_s  dxds  \right]  \right)^{1/2}
\end{split}
\end{equation*}
For the first term on the right hand side, we have
$$
\mathbb E\left[ \int_0^T \int_{\mathbb{R}^n}|\Delta(s,\mathbb X_s^{x},u_s)|^2 \varphi(\frac{x}{N})dxds  \right]\le  \int_{|x|\le N+2} \mathbb E\left[  \int_0^T |\Delta(s,\mathbb X_s^{x},u_s)|^2ds\right]dx<\infty.
$$
Note that $\tilde{\mathbb X}_s^{\mathbb X_s^x}=x$. Hence
$
\partial_y\tilde{\mathbb X}_s^{y}|_{y=\mathbb X_s^x}\partial_x\mathbb X_s^x=I
$,
and thus $
{\color{black}|\det\partial_y\tilde{\mathbb X}_s^y|_{y= \mathbb X_s^{x}}|^{-1}=|\det\partial_x\mathbb X_s^{x}|}
$.
For the second term, it holds that
\begin{equation*}
\begin{split}
&\mathbb E\left[ \int_0^T \int_{\mathbb{R}^n}\xi^2(\mathbb X_s^{x})\varphi(\frac{x}{N})|\det\partial_y\tilde{\mathbb X}_s^y|_{y=\mathbb X^x_s}|^{-2}p^2_s  dxds  \right]\\
\le & C\mathbb E\left[ \int_0^T\int_{\mathbb{R}^n}\varphi(\frac{x}{N})|\det\partial_x\mathbb X_s^{x}|^{2}dxds  \right]\\
\le &C \int_{|x| \le N+2} \mathbb E\left[ \int_0^T |\det \partial_x\mathbb X_s^{x}|^2ds  \right]dx <\infty
\end{split}
\end{equation*}
Thus, we see that
$
\mathbb E\left[ \int_0^T \int_{\mathbb{R}^n} \Delta (s,\mathbb X_s^{x},u_s)g(s,x)dxds  \right]<\infty
$. Then we have
\begin{equation*}
\begin{split}
0 \ge &\mathbb E\left[ \int_0^T \int_{\mathbb{R}^n} \Delta (s,\mathbb X_s^{x},u_s)g(s,x)dxds  \right]\\
=&\mathbb E\left[ \int_0^T \int_{\mathbb{R}^n} \Delta (s,\mathbb X_s^{x},u_s)\xi(\mathbb X_s^{x})\varphi(\frac{x}{N})|\det \partial_y \tilde{\mathbb X}_s^y|_{y=\mathbb X^x_s}|^{-1}p_s dxds  \right]\\
=&{\color{black}\mathbb E\left[ \int_0^T \int_{\mathbb{R}^n} \Delta (s,x,u_s)\xi(x)\varphi(\frac{\tilde{\mathbb X}_s^x}{N})p_s dxds  \right],}
\end{split}
\end{equation*}
where we apply the change of variable 
from the second to the third line in the above.\end{color} As $N \rightarrow +\infty$, it reduces to
$$
\mathbb E\left[ \int_0^T \int_{\mathbb{R}^n} \Delta (s,x,u_s)\xi(x) p_s dxds  \right] \le 0.
$$
From the arbitrariness of $\xi$, $p$ and $u$, we have that
\begin{equation}\label{inf_less}
\sup_u \Delta (s,x,u) \le 0\ \ \ {\rm for\ all}\ x\in\mathbb{R}^n,\ {\rm a.e.}\ s\in[0,T],\ {\rm a.s.}
\end{equation}

Next, we show that the equality holds. Since the optimal control $u^{*,t,x}_\cdot$ and its corresponding state denoted by $X^{u^{*,t,x}}_\cdot$ exist. For simplicity, we abbreviate $(X^{u^{*,s,x}}_\cdot,u^{*,s,x}_\cdot)$ as $(X^*_\cdot,u^*_\cdot)$. It follows from (\ref{DMZ5}) that
$$
{\color{black}\Delta(s,X^{*}_s;u^{*}_s)=0,\text{ for a.e. $s \in[t,T]$, a.s..}}
$$
Denote by $\Delta(s,x):=\sup_u \Delta(s,x,u)$. Then, we see that
$$
\Delta(s,x,0) \le \Delta(s,x) \le 0.
$$
This implies that
$$
|\Delta(s,x)|\le |\Delta(s,x,0)|\le CK_t(1+|x|^2),
$$
which further yields that $\Delta(\cdot,x) \in {\color{black}M_{\mathscr{F}}^2(L^2(0,T;\mathbb R^-))}$ for any $x$. Let $\zeta(t)$ be a mollifier defined on $[0,+\infty)$, i.e.
$$\zeta(t)=\left \{\begin{aligned}C\exp(-\frac{1}{1-t^2}), &\text{ if $ t \le 1$;}\\ 0,\qquad &\text{ otherwise;}\end{aligned}
\right .$$
with the constant $C$ selected so that $\int_0^{\infty} \zeta(t)dt=1$ and $\zeta_n(t)=n\zeta(nt)$. Define
$$
\Delta_n(s,x)=\int_0^{\infty} \zeta_n(u)\Delta(s+u,x)du.
$$
We shall have that
\begin{equation}\label{ineq_delta}
\mathbb E\left[  \int_0^T \Delta_n(s,x)ds \right] \rightarrow \mathbb E\left[  \int_0^T \Delta(s,x)ds \right]
\end{equation}
as $n \rightarrow +\infty$. Note that
\begin{equation*}
\begin{split}
\Delta_n(s,x)=&\int_0^{\infty} \zeta_n(u)\Delta(s+u,x)du\\
 \ge& \int_0^{\infty} \zeta_n(u)\Delta(s+u,x,u^{*}_{s+u})du\\
=& \int_0^{\infty} \zeta_n(u)(\Delta(s+u,x,u^{*}_{s+u})-\Delta(s+u,X^{*}_{s+u},u^{*}_{s+u}))du
\end{split}
\end{equation*}
From the assumption of the theorem, we see that
$$
|\Delta(s+u,x,u^{*}_{s+u})-\Delta(s+u,X^{*}_{s+u},u^{*}_{s+u})|\le CK_{s+u}(1+|x|+|X^*_{s+u}|+|u^*_{s+u}|)|X^*_{s+u}-x|.
$$
Hence,
\begin{equation*}
\begin{split}
&\mathbb E\left| \int_0^{\infty} \zeta_n(u)(\Delta(s+u,x,u^{*}_{s+u})-\Delta(s+u,X^{*}_{s+u},u^{*}_{s+u}))du \right|\\
\le& C\left(\mathbb E\left[\int_0^{\infty} \zeta_n(u)  K_{s+u}(1+|x|+|X^*_{s+u}|+|u^*_{s+u}|)^2du \right]\right)^{1/2}\\
&\left( \mathbb E\left[   \int_0^{\infty} \zeta_n(u)  K_{s+u} |X^*_{s+u}-x|^2du\right] \right)^{1/2}\\
\le& C\left(\mathbb E\left[\int_0^{\infty} \zeta_n(u)  K^2_{s+u}du \right]\right)^{1/2} \left(\mathbb E\left[\int_0^{\infty} \zeta_n(u) (1+|x|+|X^*_{s+u}|+|u^*_{s+u}|)^4du \right]\right)^{1/4}\\
&\left( \mathbb E\left[   \int_0^{\infty} \zeta_n(u) |X^*_{s+u}-x|^4du\right] \right)^{1/4}
\end{split}
\end{equation*}
Then, we see that, for all  $s$,
$$
\mathbb E\left[   \int_0^{\infty} \zeta_n(u) |X^*_{s+u}-x|^4du\right] \rightarrow 0
$$
and
$$
E\left[\int_0^{\infty} \zeta_n(u) (1+|x|+|X^*_{s+u}|+|u^*_{s+u}|)^4du \right]
$$
is uniformly bounded with respect to $n$. Moreover, it holds that, for almost all $s$,
$$
\mathbb E\left[   \int_0^{\infty} \zeta_n(u)  K^2_{s+u}du\right] \rightarrow \mathbb E\left[  K_s^2\right].
$$
Hence, for almost $s$,
$$
\liminf_{n} \mathbb E\left[  \Delta_n(s,x) \right]\ge 0.
 $$
 From \eqref{ineq_delta}, we have
 $$
 \mathbb E\left[  \int_0^T \Delta(s,x)ds \right] \ge 0.
 $$
 Combining with the fact that $\Delta (s,x) \le 0$, we obtain that
 $$
 \Delta (s,x)=0.
 $$
\begin{flushright}
	\qed
\end{flushright}
\end{proof}
In above, we have proved that the value function is the solution of the stochastic HJB equation under suitable conditions. Next, we will prove a converse result.
\begin{prop}\label{prop_svt}[Stochastic Verification Theorem]
Let $(\Phi,\Psi)$ be the solution of stochastic HJB equation \eqref{mz4} and  assume that they satisfy the regularity assumptions in Proposition \ref{pro:2.1}. Then, for any $(t,x)$ and admissible control $u$, we have
$$
V(t,x) \le J(t,x;u).
$$
Moreover, if there exists an admissible control $ u$ such that, for almost all $s\in[t,T]$,
\begin{equation*}
\begin{split}
&G\big(s,X^{t,x;u}_s,V(t,X^{t,x; u}),\Psi(t,X^{t,x; u}),V_{x}(t,X^{t,x; u}),\Psi
_{x}(t,X^{t,x; u}),V_{xx}(t,X^{t,x; u}), u_s\big)\\
= &\inf_{v} G\big(s,X^{t,x;u}_s,V(t,X^{t,x; u}),\Psi(t,X^{t,x;u}),V_{x}(t,X^{t,x;u}),\Psi
_{x}(t,X^{t,x;u}),V_{xx}(t,X^{t,x;u}),v\big), a.e.,
\end{split}
\end{equation*}
then $ u$ is the optimal control.
\end{prop}
\begin{proof}
The result is obtained by applying It\^o formula to $V(s,X^{t,x;u}_s)$ and comparing it with $Y_s^{t,x;u}$. Since the calculation is almost the same to previous proposition, we omit the proof here.	
\begin{flushright}
	\qed
\end{flushright}	
\end{proof}

\section{The Maximum Principle of  Stochastic Recursive Control Problem}

In this section, we derive the
 stochastic maximum principle of Problem \ref{pro:2.1}. We first  define the
Hamiltonian function $H: \Omega \times[0,T] \times \mathbb{R}^n\times \mathbb{R}\times \mathbb{R%
}\times \mathbb{R}^n\times \mathbb{R}^{n}\times \mathbb{R} \times U%
\rightarrow \mathbb{R}$ by
\begin{equation*}
\begin{array}{ll}
\displaystyle  H(t,x,y,z,p,q,k,u) = \langle p, b(t,x,u)\rangle
+\langle q, \sigma(t,x,u)\rangle-kf(t,x,y,z,u).
\end{array}
\label{eq:4.2}
\end{equation*}

To simplify our argument, we introduce some abbreviated
notations. Now, let $(\bar{u}; \bar{X},\bar{Y},\bar{Z})$  be an optimal pair of Problem
\ref{pro:2.1}. For $\varphi =b,\sigma, b_{x},b_u, 
\sigma_{x},\sigma_u,$ define
\begin{eqnarray*}
\bar{\varphi}(t):=\varphi(t, \bar{X}_t, \bar{u}_t),
\end{eqnarray*}
for $\varphi =f, f_x, f_y,f_z,f_u$,
\begin{eqnarray*}
\bar{\varphi}(t):=\varphi(t,\bar{X}_t, \bar {Y}_t, \bar {Z}_t,
\bar{u}_t),
\end{eqnarray*}
and for $h$,
\begin{equation*}
\bar h(T):=h(\bar X_T),\ \ \bar h_x(T):=h_x(\bar X_T).
\end{equation*}%

Now we are ready to give the necessary conditions of optimality for the optimal control of Problem \ref{pro:2.1}. Let $(\bar{u}; \bar{\Theta})=(\bar{u}; \bar{X},\bar{Y}%
,\bar{Z})$ be an optimal 4-tuple. Fix any admissible control $u \in {\cal U}^2[0,T]$. Consider $u^1 \in M_{\mathscr{F}}^\infty(0,T;\mathbb R^k)$ as $u^1_t=\frac{u_t-\bar u_t}{|u_t-\bar u_t| \vee 1} $. For any
 $\varepsilon \in [0,1],$ we construct a perturbed admissible control as below
 \begin{eqnarray*}
 	u^{\varepsilon }=\bar
 	u+\varepsilon u^1.
 \end{eqnarray*}%
It is easy to see that $u^\varepsilon$ is also an admissible control.
Denote by $\left(
X^{\varepsilon },Y^{\varepsilon },Z^{\varepsilon}\right)$  the corresponding state equation and consider the
following variational equations:
\begin{numcases}{}\label{eq:4.7}
dX^1_t=\bigg[\bar b_x(t)X^1_t+\bar
b_u(t){u}^1_t\bigg]dt+\displaystyle\bigg[\bar
\sigma_x(t)X^1_t +\bar \sigma_u(t){u}^1_t\bigg]dW_t,
\nonumber\\
dY^1(t)=-\bigg[\bar f_x(t)X^1_t+\bar f_y(t) Y^1_t+\bar
f_z(t)Z^1_t+\bar f_u(t){u}^1_t\bigg]dt
+\displaystyle Z^{1}_tdW_t,\nonumber\\
X^1_0=0,\nonumber\\ Y_T=\bar h_x(T)X^1_T.
\end{numcases}
Since $h_x$ is of linear growth with respect to $x$, the terminal $\bar h_x(T)X^1_T$ is not $L^2$-integrable in general. Thus, the solvability of \eqref{eq:4.7} is not obvious. For that purpose, we shall introduce the following result for BSDE with $L^p$-terminal. It has been proved in \cite{briand2003lp}.
\begin{lem}\label{lp_BSDE}
Consider the following BSDE
\begin{equation*}
\displaystyle\left\{
\begin{array}{lll}
dY_t=-f(t,Y_t,Z_t)dt+Z_t dW_t,\\
Y_T=\xi,
\end{array}
\right.
\end{equation*}
with $f$ is uniformly Lipschitz continuous with respect to $(y,z)$ and $\xi$ is $L^p$-integrable with some $p>1$. There exists a unique solution $(Y,Z)$, and for some constant $\tilde{C}$,
$$
\|Y\|^p_{\mathcal S^p}+\|Z\|^p_{M^p} \le \tilde{C}\mathbb E\left[ |\xi|^p+\left(\int_0^T |f(t,0,0)|dt\right)^p \right].
$$
\end{lem}
We shall have the following lemmas.
\begin{lem} \label{lem:3.2}
	Under Assumptions \ref{assum_b}, it holds that
	\begin{eqnarray}\label{eq:4.9}
	\mathbb E\sup_{0\leq t\leq
		T}|X^\varepsilon_t-\bar{X}_t|^p=O(\varepsilon^p),
	\end{eqnarray}
	and
	\begin{eqnarray}\label{SDE_appro_X_2}
	\mathbb E\sup_{0\leq t\leq
		T}|X^\varepsilon_t-\bar{X}_t-\varepsilon X^1_t|^p=o(\varepsilon^p),
	\end{eqnarray}
for any $p>1$.	
\end{lem}

\begin{proof}
	The proof is rather standard.
	For (\ref{eq:4.9}), by the $L^p$ estimate for SDE (see Proposition 2.1 in \cite{mou2007variational}) and Assumptions \ref{assum_b}, we have
\begin{eqnarray*}
		\mathbb E\bigg(\sup_{0\leq t\leq T}|X^\varepsilon_t-\bar X_t|^p\bigg)
		&\leq & C\bigg[\mathbb E\bigg(\int_0^T|b(t,\bar X_t,u^\varepsilon_t)-b(t,\bar X_t, \bar u_t)
		|dt\bigg)^p
		\\&&\ \ \ \ \ \ +\mathbb E\bigg(\int_0^T|\sigma(t,\bar X_t,u^\varepsilon_t)
		-\sigma(t,\bar X_t,\bar u_t)|^2dt\bigg)^{p/2}\bigg]
		\\ &\leq & C\mathbb E\bigg(\int_0^T|u^\varepsilon_t-\bar u_t
		|^2dt\bigg)^{p/2}
		\\&=& C\mathbb E\bigg(\int_0^T|\varepsilon u^1_t|^2dt\bigg)^{p/2}
		\\&=& C\varepsilon^p\mathbb E\bigg(\int_0^T| u^1_t|^2dt\bigg)^{p/2}
		\\&=& O(\varepsilon^p).
	\end{eqnarray*}
For \eqref{SDE_appro_X_2}, denote $\delta X:=X^{\varepsilon}-\bar X-\varepsilon X^1$. Then, we have
\begin{equation*}
\displaystyle\left\{
\begin{array}{lll}
d\delta X_t= \bar b_x(t) \delta X_t+(\tilde b_x(t)-\bar b_x(t))(X^{\varepsilon}_t-\bar X_t)dt+ \bar \sigma_x(t) \delta X_t+(\tilde \sigma_x(t)-\bar \sigma_x(t))(X^{\varepsilon}_t-\bar X_t)dW_t,\\
\delta X_0=0,
\end{array}
\right.
\end{equation*}
	with
	$$
	\tilde b_x(t):=\int_0^1 b_x(t,\bar X_t+\lambda (X^{\varepsilon}_t-\bar X_t),\bar u_t+\lambda \varepsilon u^1_t)d\lambda
	$$
	and
	$$
	\tilde \sigma_x(t):=\int_0^1 \sigma_x(t,\bar X_t+\lambda (X^{\varepsilon}_t-\bar X_t),\bar u_t+\lambda \varepsilon u^1_t)d\lambda.
	$$
	From previous estimation for $X^{\varepsilon}-\bar X$ and the standard estimation for SDEs, we shall have \eqref{SDE_appro_X_2}. The proof is completed.
	\begin{flushright}
		\qed
	\end{flushright}
\end{proof}
It is also easy to show that $X^1_T$ is $L^p$-integrable for any $p>1$, which implies that the terminal $\bar h_x(T)X^1_T$ is $L^p$-integrable for any $p \in (1,2)$. Combining Lemma \ref{lp_BSDE}, we shall have
\begin{lem}\label{lem_existence}
Under Assumptions \ref{assum_b} and \ref{assum_f}, FBSDE \eqref{eq:4.7} admits a unique solution $(X^1,Y^1,Z^1)$. Moreover, $X^1 \in S^{p_1}$ and $(Y^1,Z^1) \in S^{p_2} \times M^{p_2}$ for any $p_1>1$ and any $p_2 \in(1,2)$.
\end{lem}
Next, we prove the following expansion for $\bar Y$.
\begin{lem} \label{lem:4.3}
	Under Assumptions \ref{assum_b} and \ref{assum_f}, we have for any $p\in(1,2)$
	\begin{eqnarray}\label{appro_Y}
	\lim_{\varepsilon\rightarrow 0}\mathbb E&\displaystyle \sup_{0\leq t\leq
		T}|\frac {Y^\varepsilon_t-\bar{Y}_t}{\varepsilon}-
	Y^1_t|^p=0.
	\end{eqnarray}
\end{lem}
\begin{proof}
A direct calculation gives
	\begin{equation*}
	\begin{split}
	&Y^\varepsilon_t-\bar{Y}_t-\varepsilon Y^1_t\\=&\tilde h_x(T)(X^{\varepsilon}_T-\bar X_T)+\bar h_x(T)\delta X(T)\\
	&
\int_t^T \tilde f_x(s)(X^\varepsilon_s-\bar{X}_s-\varepsilon X^1_s)ds
+\int_t^T \tilde f_y(s)(Y^\varepsilon_s-\bar{Y}_s-\varepsilon Y^1_s)ds
\\& +\int_t^T \tilde f_z(s)(Z^\varepsilon_s-\bar{Z}_s-\varepsilon Z^1_s)ds
+\int_t^T (\tilde f_x(s)-\bar f_x(s))\varepsilon X^1_sds
\\&+\int_t^T(\tilde f_y(s)-\bar f_y(s))\varepsilon Y^1_sds
+\int_t^T(\tilde f_z(s)- \bar f_z(s))\varepsilon Z^1_sds
\\&+\varepsilon\int_0^t (\tilde {f}_u(s)-\bar f_u(s))u^1_sds
+\int_t^T (Z^\varepsilon(s)-\bar{Z}(s)-\varepsilon Z^1_s)dW_s,
\end{split}
	\end{equation*}
	with
$$
	\tilde f_x(t):=\int_0^1 f_x(t,\bar X_t+\lambda (X^{\varepsilon}_t-\bar X_t),\bar Y(t)+\lambda (Y^{\varepsilon}_t-\bar Y_t),\bar Z_t+\lambda (Z^{\varepsilon}_t-\bar Z_t),\bar u_t+\lambda \varepsilon u^1_t)d\lambda,
$$
and $\tilde f_y,\tilde f_z$ and $\tilde h_x$ similarly defined. Combining Lemma \ref{lp_BSDE}  and Lemma \ref{lem_existence}, we have
\begin{eqnarray*}
\mathbb E&\displaystyle \sup_{0\leq t\leq
	T}|Y^\varepsilon_t-\bar{Y}_t-\varepsilon
Y^1_t|^p=o(\varepsilon^p),
\end{eqnarray*}
which is equivalent to \eqref{appro_Y}.
	\begin{flushright}
		\qed
	\end{flushright}
\end{proof}
Finally, we shall have the following maximum principle.
\begin{thm}\label{thm:4.5}
Under Assumptions \ref{assum_b} and \ref{assum_f},
set $(\bar u;\bar{\Theta})=( \bar u;\bar{X},\bar{Y},\bar{Z})$ be an optimal 4-tuple of Problem %
\ref{pro:2.1}.
 Then, we have, for a.e. $t \in [0,T]$, almost surely
\begin{equation}\label{eq:44}
 H_u(t,\bar{X}_t, \bar{Y}_t,\bar{Z}_t,p_t, q_t,k_t,\bar{u}_t)(u-\bar u_t)\ge 0, \text{for any $u \in U$,}
\end{equation}
where $\Lambda=(p,q,k)$ is the
solution to the following  FBSDE:
\begin{equation}\label{eq:4.6}
\left\{\begin{array}{ll} dp_t=-\bar H_x(t)dt+\displaystyle
q_tdW_t,\\
dk_t=-\bar H_y(t)dt -\displaystyle\bar
H_{z}(t)dW_t,
\\
p_T=-\bar h_x^*(T)k_T,\\
k_0=-1,~~~~0\leq t\leq T,
\end{array}\right.
\end{equation}
with $\eta(t)=H, H_{x}, H_{y}, H_{z},H_{u},$ defined as
$$\bar\eta(t):=\eta(t,\bar{\Theta}_t,\Lambda _t,\bar u_t).$$
\end{thm}

\begin{proof}
	Fix any admissible control $u \in {\cal U}^\infty[0,T]$. For any
	$\varepsilon \in [0,1],$ we construct a perturbed admissible control
	\begin{eqnarray*}
		u^{\varepsilon }=\bar
		u+\varepsilon u^1,
	\end{eqnarray*}
with $u^1_t=\frac{u_t-\bar u_t}{|u_t-\bar u_t|\vee 1} $ and the corresponding state equation is denoted by $\left(
	X^{\varepsilon },Y^{\varepsilon },Z^{\varepsilon}\right)$.  Let $(X^1,Y^1,Z^1)$ be the solution of FBSDE \eqref{eq:4.7}. From Lemma \ref{lem:4.3}, we have for any $p \in(0,1)$,
	\begin{eqnarray}
	\lim_{\varepsilon\longrightarrow 0}\mathbb E&\displaystyle \left[\sup_{0\leq t\leq
		T}\left|\frac {Y^\varepsilon_t-\bar{Y}_t}{\varepsilon}-
	Y^1_t\right|^p\right]=0.
	\end{eqnarray}
	Then, it holds that
	\begin{equation}\label{eq:4.4}
	Y^1_0=\lim_{\varepsilon\rightarrow
		0^+}\frac{Y^\varepsilon_0-\bar Y_0}{\varepsilon}=\lim_{\varepsilon\rightarrow
		0^+}\frac{J(0,x,u^\varepsilon)-J(0,x,\bar{u})}{\varepsilon}\ge 0.
\end{equation}
Applying It\^{o} formula to
	$\langle Y^1_t, k_t \rangle+ \langle X^1_t, p_t \rangle,$ we
	have
	$$
	\begin{array}{ll}
	Y^1_0=&E\displaystyle\int_0^TH_u(t,\bar{X}_t,\bar{Y}_t,\bar{Z}_t,\bar{u}_t,p_t,q_t,k_t)
	u^1_tdt.
	\end{array}
	$$
	Thus, by the variational inequality (\ref{eq:4.4}), we have
	$$
	E\displaystyle\int_0^TH_u(t,\bar{X}_t,\bar{Y}_t,\bar{Z}_t,\bar{u}_t,p_t,q_t,k_t)u^1_tdt\geq
	0,
	$$
	which is equivalent to 
	$$
	E\displaystyle\int_0^TH_u(t,\bar{X}_t,\bar{Y}_t,\bar{Z}_t,\bar{u}_t,p_t,q_t,k_t)\frac{u_t-\bar u_t}{|u_t-\bar u_t|\vee 1}dt\geq
	0,
	$$
	for any $u \in U^2[0,T]$.
   Due to the arbitrariness of $u1$, we shall get that 
   $$
   H_u(t,\bar{X}_t,\bar{Y}_t,\bar{Z}_t,\bar{u}_t,p_t,q_t,k_t)\frac{u-\bar u_t}{|u-\bar u_t|\vee 1} \ge 0,
   $$
for any $u \in U$. This will implies \eqref{eq:44}.
\begin{flushright}
		\qed
\end{flushright}
\end{proof}

\section{The Relationship between SMP and DPP}
In this section, we will state the relation between SMP and DPP for the recursive utility setup.
\begin{thm}\label{mz27}
We assume that the value function admits the following form
\begin{eqnarray}\label{mz121}
V(t,x) =h(x)+\int_{t}^{T}\Gamma(s,x)ds-\int_{t}^{T}\Psi(s,x)dW_{s},\ \ \ t\in[0,T],
\end{eqnarray}
where for a.e. $s\in[t,T]$ a.s. $\omega\in\Omega$,
\begin{eqnarray}\label{mz11}
\Gamma(s,\bar{X}_s^{t,x})&=&G\big(s,\bar{X}_s^{t,x},\bar{u}_s,V(s,\bar{X}_s^{t,x}),\Psi(s,\bar{X}_s^{t,x}),V_x(s,\bar{X}_s^{t,x}),\Psi_x(s,\bar{X}_s^{t,x}),V_{xx}(s,\bar{X}_s^{t,x})\big)\\
&=&\inf_{u\in U}G\big(s,\bar{X}_s^{t,x},u,V(s,\bar{X}_s^{t,x}),\Psi(s,\bar{X}_s^{t,x}),V_x(s,\bar{X}_s^{t,x}),\Psi_x(s,\bar{X}_s^{t,x}),V_{xx}(s,\bar{X}_s^{t,x})\big).\nonumber
\end{eqnarray}
If $V\in C^{1,3}([0,T]\times\mathbb{R}^n)$ and $\Gamma_{x},\Psi_x\in C^{0,0}([0,T]\times\mathbb{R}^n)$, we have
\begin{eqnarray}\label{mz13}
p_s&=&-V_x(s,\bar{X}_s^{t,x})k_s,\nonumber\\
q_s&=&-\bigg[V_{xx}(s,\bar{X}_s^{t,x})\sigma(s,\bar{X}_s^{t,x},\bar{u}_s)\\
&&\ \ \ \ +V_{x}(s,\bar{X}_s^{t,x})f_z\big(s,\bar{X}_s^{t,x},V(s,\bar{X}_s^{t,x}),\sigma^*V_x(s,\bar{X}_s^{t,x})+\Psi(s,\bar{X}_s^{t,x}),\bar{u}_s\big)+\Psi_x(s,\bar{X}_s^{t,x})\bigg]k_s,\nonumber
\end{eqnarray}
for a.e. $s\in[0,T]$ a.s., where $k_s$ satisfies $k_0=-1$ and
\begin{equation}\label{SMP&DPP_k}
\begin{split}
dk_s= & f_y(s,\bar X_s^{t,x},\bar u_s,V(s,\bar X_s^{t,x}),\sigma(s,\bar X_s^{t,x})V_x(s,\bar X_s^{t,x})+\Psi(s,\bar X_s^{t,x}))k_sds\\
& f_z(s,\bar X_s^{t,x},\bar u_s,V(s,\bar X_s^{t,x}),\sigma(s,\bar X_s^{t,x})V_x(s,\bar X_s^{t,x})+\Psi(s,\bar X_s^{t,x}))k_sdW_s, \text{ for  $s\in [t,T]$.}
\end{split}
\end{equation}
\end{thm}
\begin{proof} 
First note that there exists a unique solution of \eqref{SMP&DPP_k}, since $f$ is Lipschitz continuous with respect to $y$ and $z$.
Noticing the first equality in (\ref{mz11}), we know
\begin{eqnarray*}
G\big(s,\bar{X}_s^{t,x},\bar{u}_s,V(s,\bar{X}_s^{t,x}),\Psi(s,\bar{X}_s^{t,x}),V_x(s,\bar{X}_s^{t,x}),\Psi_x(s,\bar{X}_s^{t,x}),V_{xx}(s,\bar{X}_s^{t,x})\big)-\Gamma(s,\bar{X}_s^{t,x})=0
\end{eqnarray*}
Then, since $V$ satisfies HJB equation (\ref{mz4}) and has a form as (\ref{mz12}), we conclude
\begin{eqnarray*}
\Gamma(s,x)&=&\inf_{u\in U}G\big(s,x,u,V(s,x),\Psi(s,x),V_x(s,x),\Psi_x(s,x),V_{xx}(s,x)\big)\nonumber\\
&\leq&G\big(s,x,\bar{u}_s,V(s,x),\Psi(s,x),V_x(s,x),\Psi_x(s,x),V_{xx}(s,x)\big).
\end{eqnarray*}
Thus
\begin{eqnarray*}
0&=&G\big(s,\bar{X}_s^{t,x},\bar{u}_s,V(s,\bar{X}_s^{t,x}),\Psi(s,\bar{X}_s^{t,x}),V_x(s,\bar{X}_s^{t,x}),\Psi_x(s,\bar{X}_s^{t,x}),V_{xx}(s,\bar{X}_s^{t,x})\big)-\Gamma(s,\bar{X}_s^{t,x})\nonumber\\
&\leq&G\big(s,x,\bar{u}_s,V(s,x),\Psi(s,x),V_x(s,x),\Psi_x(s,x),V_{xx}(s,x)\big)-\Gamma(s,x).
\end{eqnarray*}
Bearing in mind that $V\in C^{1,3}([0,T]\times\mathbb{R}^n)$ and $\Gamma_{x}\in C^{0,0}([0,T]\times\mathbb{R}^n)$, we have
\begin{eqnarray*}
{{\partial}\over{\partial x}}\Big\{G\big(s,x,\bar{u}_s,V(s,x),\Psi(s,x),V_x(s,x),\Psi_x(s,x),V_{xx}(s,x)\big)-\Gamma(s,x)\Big\}_{x=\bar{X}_s^{t,x}}=0.
\end{eqnarray*}
This implies
\begin{eqnarray}\label{mz16}
&&(\sigma_x)^*(s,\bar{X}_s^{t,x},\bar{u}_s)\big(V_{xx}(s,\bar{X}_s^{t,x})\sigma(s,\bar{X}_s^{t,x},\bar{u}_s)\big)+{1\over2}tr\big((\sigma\sigma^*)(s,\bar{X}_s^{t,x},\bar{u}_s)V_{xxx}(s,\bar{X}_s^{t,x})\big)\nonumber\\
&&+b^*_x(s,\bar{X}_s^{t,x},\bar{u}_s)V_x(s,\bar{X}_s^{t,x})+V_{xx}(s,\bar{X}_s^{t,x})b(s,\bar{X}_s^{t,x},\bar{u}_s)
+\sigma^*_x(s,\bar{X}_s^{t,x},\bar{u}_s)\Psi_x(s,\bar{X}_s^{t,x})\nonumber\\
&&+\Psi_{xx}(s,\bar{X}_s^{t,x})\sigma(s,\bar{X}_s^{t,x},\bar{u}_s)+f_x\big(s,\bar{X}_s^{t,x},V(s,\bar{X}_s^{t,x}),\Psi(s,\bar{X}_s^{t,x})+\sigma^*V_x(s,\bar{X}_s^{t,x}),\bar{u}_s\big)\nonumber\\
&&+f_y\big(s,\bar{X}_s^{t,x},V(s,\bar{X}_s^{t,x}),\Psi(s,\bar{X}_s^{t,x})+\sigma^*V_x(s,\bar{X}_s^{t,x}),\bar{u}_s\big)V_x(s,\bar{X}_s^{t,x})\\
&&+f_z\big(s,\bar{X}_s^{t,x},V(s,\bar{X}_s^{t,x}),\Psi(s,\bar{X}_s^{t,x})+\sigma^*V_x(s,\bar{X}_s^{t,x}),\bar{u}_s\big)\Psi_{x}(s,\bar{X}_s^{t,x})\nonumber\\
&&+f_z\big(s,\bar{X}_s^{t,x},V(s,\bar{X}_s^{t,x}),\Psi(s,\bar{X}_s^{t,x})+\sigma^*V_x(s,\bar{X}_s^{t,x}),\bar{u}_s\big)\sigma^*_x(s,\bar{X}_s^{t,x},\bar{u}_s)V_{x}(s,\bar{X}_s^{t,x})\nonumber\\
&&+f_z\big(s,\bar{X}_s^{t,x},V(s,\bar{X}_s^{t,x}),\Psi(s,\bar{X}_s^{t,x})+\sigma^*V_x(s,\bar{X}_s^{t,x}),\bar{u}_s\big)\sigma^*(s,\bar{X}_s^{t,x},\bar{u}_s)V_{xx}(s,\bar{X}_s^{t,x})-\Gamma_x(s,\bar{X}_s^{t,x})=0.\nonumber
\end{eqnarray}
Here and in the rest of this paper, $$\displaystyle{1\over2}tr\big((\sigma\sigma^*)V_{xxx}\big)\triangleq\bigg(tr\big(\sigma\sigma^*(V_x)^1_{xx}\big),
tr\big(\sigma\sigma^*(V_x)^2_{xx}\big),\cdots,tr\big(\sigma\sigma^*(V_x)^n_{xx}\big)\bigg)^*.$$
On the other hand, from (\ref{mz121}), we have
\begin{eqnarray*}
V_x(t,x)=h_x(x)+\int_t^T\Gamma_x(s,x)ds-\int_t^T\Psi_x(s,x)dW_s,\ \ \ t\in[0,T].
\end{eqnarray*}
Then by the application of It\^{o}'s formula to $-V_x(s,\bar{X}_s^{t,x})k_s$, it turns out, from (\ref{mz16}), that
\begin{eqnarray*}\label{mz15}
&&-V_x(s,\bar{X}_s^{t,x})k_s\nonumber\\
&=&-V_x(T,\bar{X}_T^{t,x})k_T+\int_s^Tk_rdV_x(r,\bar{X}_r^{t,x})
+\int_s^TV_x(r,\bar{X}_r^{t,x})dk_r+\int_s^TdV_x(r,\bar{X}_r^{t,x})\cdot dk_r\nonumber\\
&=&-V_x(T,\bar{X}_T^{t,x})k_T+\int_s^T\bigg[-\Gamma_x(r,\bar{X}_r^{t,x})+{1\over2}tr\big((\sigma\sigma^*)(r,\bar{X}_r^{t,x},\bar{u}_r)V_{xxx}(r,\bar{X}_r^{t,x})\big)\nonumber\\
&&\ \ \ \ \ \ \ \ \ \ \ \ \ \ \ \ \ \ \ \ \ \ \ \ \ \ \ \ \ \ \ \ +V_{xx}(r,\bar{X}_r^{t,x})b(r,\bar{X}_r^{t,x},\bar{u}_r)+\Psi_{xx}(r,\bar{X}_r^{t,x})\sigma(r,\bar{X}_r^{t,x},\bar{u}_r)\bigg]k_rdr\nonumber\\
&&+\int_s^T\bigg[\Psi_x(r,\bar{X}_r^{t,x})+V_{xx}(r,\bar{X}_r^{t,x})\sigma(r,\bar{X}_r^{t,x},\bar{u}_r)\bigg]k_rdW_r\nonumber\\
&&+\int_s^TV_x(r,\bar{X}_r^{t,x})f_{y}(r,\bar{X}_r^{t,x},\bar{Y}_r^{t,x},\bar{Z}_r^{t,x},\bar{u}_r)k_rdr+\int_s^TV_x(r,\bar{X}_r^{t,x})f_{z}(r,\bar{X}_r^{t,x},\bar{Y}_r^{t,x},\bar{Z}_r^{t,x},\bar{u}_r)k_rdW_r\nonumber\\
&&+\int_s^T\bigg[\Psi_x(r,\bar{X}_r^{t,x})+V_{xx}(r,\bar{X}_r^{t,x})\sigma(r,\bar{X}_r^{t,x},\bar{u}_r)\bigg]f_{z}(r,\bar{X}_r^{t,x},\bar{Y}_r^{t,x},\bar{Z}_r^{t,x},\bar{u}_r)k_r dr\nonumber\\
&=&-V_x(T,\bar{X}_T^{t,x})k_T\nonumber\\
&&+\int_s^T\bigg[-(\sigma_x)^*(r,\bar{X}_r^{t,x},\bar{u}_r)\big(V_{xx}(r,\bar{x}_r)\sigma(r,\bar{X}_r^{t,x},\bar{u}_r)\big)-b^*_x(r,\bar{X}_r^{t,x},\bar{u}_r)V_x(r,\bar{x}_r)\nonumber\\
&&\ \ \ \ \ \ \ \ \ \ \ -\sigma^*_x(r,\bar{X}_r^{t,x},\bar{u}_r)\Psi_x(r,\bar{x}_r)-f_x\big(r,\bar{X}_r^{t,x},V(r,\bar{X}_r^{t,x}),\Psi(r,\bar{X}_r^{t,x})+\sigma^*V_x(r,\bar{X}_r^{t,x}),\bar{u}_r\big)\nonumber\\
&&\ \ \ \ \ \ \ \ \ \ \ -f_y\big(r,\bar{X}_r^{t,x},V(r,\bar{X}_r^{t,x}),\Psi(r,\bar{X}_r^{t,x})+\sigma^*V_x(r,\bar{X}_r^{t,x}),\bar{u}_r\big)V_x(r,\bar{X}_r^{t,x})\nonumber\\
&&\ \ \ \ \ \ \ \ \ \ \ -f_z\big(r,\bar{X}_r^{t,x},V(r,\bar{X}_r^{t,x}),\Psi(r,\bar{X}_r^{t,x})+\sigma^*V_x(r,\bar{X}_r^{t,x}),\bar{u}_r\big)\Psi_{x}(r,\bar{X}_r^{t,x})\nonumber\\
&&\ \ \ \ \ \ \ \ \ \ \ -f_z\big(r,\bar{X}_r^{t,x},V(r,\bar{X}_r^{t,x}),\Psi(r,\bar{X}_r^{t,x})+\sigma^*V_x(r,\bar{X}_r^{t,x}),\bar{u}_r\big)\sigma^*_x(r,\bar{X}_r^{t,x},\bar{u}_r)V_{x}(r,\bar{X}_r^{t,x})\nonumber\\
&&\ \ \ \ \ \ \ \ \ \ \ -f_z\big(r,\bar{X}_r^{t,x},V(r,\bar{X}_r^{t,x}),\Psi(s,\bar{X}_r^{t,x})+\sigma^*V_x(r,\bar{X}_r^{t,x}),\bar{u}_r\big)\sigma^*(r,\bar{X}_r^{t,x},\bar{u}_r)V_{xx}(r,\bar{X}_r^{t,x})\bigg]k_rdr\nonumber\\
&&+\int_s^T\bigg[\Psi_x(r,\bar{X}_r^{t,x})+V_{xx}(r,\bar{X}_r^{t,x})\sigma(r,\bar{X}_r^{t,x},\bar{u}_r)\bigg]k_rdW_r\nonumber\\
&&+\int_s^TV_x(r,\bar{X}_r^{t,x})f_{y}(r,\bar{X}_r^{t,x},\bar{Y}_r^{t,x},\bar{Z}_r^{t,x},\bar{u}_r)k_rdr+\int_s^TV_x(r,\bar{X}_r^{t,x})f_{z}(r,\bar{X}_r^{t,x},\bar{Y}_r^{t,x},\bar{Z}_r^{t,x},\bar{u}_r)k_rdW_r\nonumber\\
&&+\int_s^T\bigg[\Psi_x(r,\bar{X}_r^{t,x})+V_{xx}(r,\bar{X}_r^{t,x})\sigma(r,\bar{X}_r^{t,x},\bar{u}_r)\bigg]f_{z}(r,\bar{X}_r^{t,x},\bar{Y}_r^{t,x},\bar{Z}_r^{t,x},\bar{u}_r)k_r dr\nonumber\\
&=&-V_x(T,\bar{x}_T)k_T\nonumber\\
&&+\int_s^T\bigg[-(\sigma_x)^*(r,\bar{X}_r^{t,x},\bar{u}_r)\big(V_{xx}(r,\bar{X}_r^{t,x})\sigma(r,\bar{X}_r^{t,x},\bar{u}_r)\big)-b^*_x(r,\bar{X}_r^{t,x},\bar{u}_r)V_x(r,\bar{X}_r^{t,x})\nonumber\\
&&\ \ \ \ \ \ \ \ \ \ \ -\sigma^*_x(r,\bar{X}_r^{t,x},\bar{u}_r)\Psi_x(r,\bar{X}_r^{t,x})-f_x\big(r,\bar{X}_r^{t,x},V(r,\bar{X}_r^{t,x}),\Psi(r,\bar{X}_r^{t,x})+\sigma^*V_x(r,\bar{X}_r^{t,x}),\bar{u}_r\big)\nonumber\\
&&\ \ \ \ \ \ \ \ \ \ \ -f_z\big(r,\bar{X}_r^{t,x},V(r,\bar{X}_r^{t,x}),\Psi(r,\bar{X}_r^{t,x})+\sigma^*V_x(r,\bar{X}_r^{t,x}),\bar{u}_r\big)\sigma^*_x(r,\bar{X}_r^{t,x},\bar{u}_r)V_{x}(r,\bar{X}_r^{t,x})
\bigg]k_rdr\nonumber\\
&&-\int_s^T-\bigg[\Psi_x(r,\bar{X}_r^{t,x})+V_{xx}(r,\bar{X}_r^{t,x})\sigma(r,\bar{X}_r^{t,x},\bar{u}_r)+V_x(r,\bar{X}_r^{t,x})f_{z}(r,\bar{X}_r^{t,x},\bar{Y}_r^{t,x},\bar{Z}_r^{t,x},\bar{u}_r)\bigg]k_rdW_r.
\end{eqnarray*}
Noticing $h_x(\bar{X}_T^{t,x})=V_x(T,\bar{X}_T^{t,x})$, by the uniqueness of the solution to FBSDE \eqref{eq:4.6}, we obtain (\ref{mz13}).
\begin{flushright}
	\qed
\end{flushright}
\end{proof}
\section{An Example: LQ Problem}
In this section, we take the LQ problem as an example to show the relationship between stochastic maximum principle and stochastic dynamical programming. Consider the following forward-backward stochastic system:
\[
\left\{
\begin{array}{l}
dX_s=\big[A_sX_s+B_su_s\big]ds+\big[C_sX_s+D_su_s\big]dW_s\\
X_t=x,\\
dY_s=-\big[\lambda_s Y_s+\langle
Q_sX_s,X_s\rangle +\langle R_su_s,u_s\rangle\big]ds+Z_sdW_s\\
Y_T=\langle GX_T,X_T\rangle.
\end{array}%
\right.
\]
The cost functional is defined as following:
\begin{eqnarray*}
J(t,x,u)=Y^{t,x,u}_t.
\end{eqnarray*}
We have the following assumptions for the coefficients.
\begin{ass}\label{assumption_lq}	
	\quad
\begin{itemize}

	\item[1.] The coefficients $A,B,C,D,\lambda,Q,$ and $R$ are all bounded $\{\mathcal F_t\}$-adapted processes;
	\item[2.] The coefficients $Q$ and $R$ are uniformly positive definitive, i.e., there exists a constant $C$ such that
	$$
	Q_s,R_s  \ge C I,\text{ for all $s\in[t,T]$, a.s.,}
	$$
where $I$ is the identity matrix.
\end{itemize}
\end{ass}
For any admissible control $u$ and initial state $x$, we introduce the corresponding adjoint equation:
\begin{eqnarray*}
\left\{
\begin{array}{l}
dp_s=-\big[A^{\ast}_sp_s+C^{\ast}_sq_s-2k_sQ_sX_s\big]ds+q_sdW_s\\
p_T=-2k_TGX_T,\\
dk_s=\lambda_s k_sds\\
k_t=-1.%
\end{array}%
\right.
\end{eqnarray*}
From the maximum principle we proved in previous section, we shall have the following theorem.
\begin{cor}\label{mz29}
If an admissible pair $(u,X)$ is the optimal pair of LQ problem, $(u,X)$ satisfies
\begin{eqnarray}\label{mz30}
-2k_s R_s u_s +D^{\ast}_sq_s+B^{\ast}_s p_s =0,
\end{eqnarray}
where $(p,q,k)$ is the solution to the corresponding adjoint equation. Therefore, the
optimal control has the dual presentation as below:
\begin{eqnarray*}
u_s ={1\over2}k^{-1}_sR^{-1}_s\big[D^{\ast}_sq_s+B^{\ast}_sp_s\big].
\end{eqnarray*}
\end{cor}

If we give an explicit presentation to $(p,q,k)$, a further expression of optimal control can be demonstrated. For this,  combining the adjoint system with the original controlled system, we have the following stochastic Hamilton system:
\begin{eqnarray*}
\left\{
\begin{array}{l}
dX_s =\big[A_sX_s+B_su_s\big]ds+\big[C_sX_s+D_su_s\big]dW_s\\
X_0=x,\\
dY_s =-\big[\lambda Y_s+\langle Q_sX_s,X_s\rangle+\langle R_su_s,u_s\rangle\big]ds+Z_sdW_s\\
Y_T =\langle GX_T,X_T\rangle,\\
dp_s =-\big[A^{\ast}_sp_s+C^{\ast}_sq_s-2k_sQ_sX_s\big]ds+q_sdW_s\\
p_T =-2k_TGX_T,\\
dk_s =\lambda_s k_sds\\
k_0=-1,\\
-2k_sR_s u_s +D^{\ast}_sq_s +B^{\ast}_s p_s =0.
\end{array}
\right.
\end{eqnarray*}
In summary, the stochastic Hamilton system completely characterizes the
optimal control in LQ problem. Therefore, solving LQ problem is equivalent to solving the stochstic
Hamilton system. But this Hamilton system consists of coupled FBSDEs. Thus, this characterization is far from satisfactory. We then introduce the  Riccati
equation to give the state feedback representation of the optimal control and further discussion of stochastic Hamilton system.

Different from the Markovian case, the Riccati equation here is a BSDE due to the non-Markovian coefficients:
\begin{eqnarray}\label{mz24}
\left\{
\begin{array}{l}
dP_s=-\{A_s^{\ast}P_s+P_sA_s+C^{\ast }_sP_s C_s+\lambda_s P_s+C^{\ast }_s L_s +L_s C_s+Q_s\\
\ \ \ \ \ \ \ \ \ \ \ \ -\big[P_sB_s+C^{\ast }_s P_s D_s +L_sD_s\big]\\
\ \ \ \ \ \ \ \ \ \ \ \ \ \ \times\big[R_s +D^{\ast}_sP_s D_s\big]^{-1}\big[P_sB_s +C^{\ast}_sP_sD_s+L_sD_s\big]^*\}ds+L_s dW_s\\
P_T=G.
\end{array}
\right.
\end{eqnarray}
The solvability of \eqref{mz24} had been studied by Tang \cite{tang2015dynamic}.
\begin{thm} Under Assumption \ref{assumption_lq}, the stochastic Riccati equation (\ref{mz24}) has a
unique solution $(P,L)$, where $P$ is a uniformly bounded and nonnegative matrix-valued process and $L$ satisfies
\begin{eqnarray*}
E\left( \int_{0}^{T}\left\vert L_{s}\right\vert ^{2}ds\right) ^{p}<\infty,
\end{eqnarray*}
for any $p>1$.
\end{thm}

For the concerned LQ problem, we still define its value function as
\begin{eqnarray*}
V(t,x)\triangleq\inf\limits_{u\in \mathcal{A}}J(t,x;u_\cdot)=\inf\limits_{u\in \mathcal{A}}Y_{t}^{t,x;u}.
\end{eqnarray*}
Then, the corresponding stochastic HJB equation is
\begin{eqnarray}\label{mz25}
V(t,x)&=&\langle Gx,x\rangle+\int_{t}^{T}\inf_{u}H\big(s,x,u,V(s,x),\Psi(s,x),V_{x}(s,x),\Psi_{x}(s,x),V_{xx}(s,x)\big)ds\nonumber\\
&&-\int_{t}^{T}\Psi(s,x)dW_{s}\\
&=&\langle Gx,x\rangle+\int_{t}^{T}\inf_{u}\{\langle V_{x}(s,x),A_sx+B_su\rangle+{1\over2}tr\big((C_sx+D_su)(C_sx+D_su)^{\ast}V_{xx}(s,x)\big)\nonumber\\
&&+\langle\Psi_{x}(s,x),C_sx+D_su\rangle+\lambda_s V(s,x)+\langle Q_sX_s,X_s\rangle+\langle R_su,u\rangle\}ds\nonumber\\
&&-\int_{t}^{T}\Psi(s,x)dW_s.\nonumber
\end{eqnarray}
With the help of stochastic Riccati equation, we can obtain a solution of above stochastic HJB equation.
\begin{prop}
If $(P,L)$ is the unique solution of the stochastic
Riccati equation (\ref{mz24}), $(\langle P_sx,x\rangle,\langle L_sx,x\rangle)$ is a
classical solution of the stochastic HJB equation (\ref{mz25}).
\end{prop}

\begin{proof}
Set
\begin{eqnarray*}
v(s,x)=\langle P_sx,x\rangle,\ \ \ \ \psi(s,x)=\langle L_sx,x\rangle.
\end{eqnarray*}
First note that
$v_{x}(s,x)=(P_s+P^*_s)x=2P_sx$, $\psi_{x}(s,x)=(L_s+L^*_s)x=2L_sx$, $v_{xx}(s,x)=P_s+P^*_s=2P_s$.
Then we have
\begin{eqnarray}\label{mz26}
&&\inf_{u}\{\langle v_{x}(s,x),A_sx+B_su\rangle+{1\over2}tr\big((C_sx+D_su)(C_sx+D_su)^{\ast}v_{xx}(s,x)\big)+\langle\psi_{x}(s,x),C_sx+D_su\rangle\nonumber\\
&&\ \ \ \ \ +\lambda_s v(s,x)+\langle Q_sx,x\rangle+\langle R_su,u\rangle\}\nonumber\\
&=&\inf_{u}\{\langle2P_sx,A_sx+B_su\rangle+{1\over2}tr\big((C_sx+D_su)(C_sx+D_su)^{\ast}2P_s\big)+\langle2L_sx,C_sx+D_su\rangle\nonumber\\
&&\ \ \ \ \ +\lambda_s\langle P_sx,x\rangle+\langle Q_sx,x\rangle+\langle R_su,u\rangle\}\nonumber\\
&=&\inf_{u}\{\langle x,P_sA_s+A_s^{\ast}P_s+Q_s+C^{\ast}_sP_sC_s+C^{\ast}_sL_s+L_sC_s+\lambda_s P_s)x\rangle\nonumber\\
&&\ \ \ \ \ +2\langle u,\big[P_sB_s+C^{\ast}_sP_sD_s+L_sD_s\big]^*x\rangle+\langle u,(R_s+D^{\ast}_sP_sD_s)u\rangle\}\\
&=&\langle\big[P_sA_s+A_s^{\ast}P_s+Q_s+C^{\ast}_sP_sC_s+C^{\ast}_sL_s+L_sC_s+\lambda_s P_s\big]x,x\rangle\nonumber\\
&&-\langle\big[P_sB_s+C^{\ast}_sP_sD_s+L_sD_s\big](R_s+D^{\ast}_sP_sD_s)^{-1}\big[P_sB_s+C^{\ast}_sP_sD_s+L_sD_s\big]^*x,x\rangle\}.\nonumber
\end{eqnarray}
Thus, noticing (\ref{mz24}), we have
\begin{eqnarray*}
&&d\langle P_sx,x\rangle\\
&=&-\{\langle\big[P_sA_s+A_s^{\ast}P_s+Q_s+C^{\ast}_sP_sC_s+C^{\ast}_sL_s+L_sC_s+\lambda_s P_s\big]x,x\rangle\\
&&\ \ \ \ -\langle\big[P_sB_s+C^{\ast}_sP_sD_s+L_sD_s\big](R_s+D^{\ast}_sP_sD_s)^{-1}\big[P_sB_s+C^{\ast}_sP_sD_s+L_sD_s\big]^*x,x\rangle\}ds\\
&&+\langle L_sx,x\rangle dW_s.
\end{eqnarray*}
By the definition for $(v,\psi)$, together with (\ref{mz26}), it turns out that
\begin{eqnarray*}
&&dv(s,x)\nonumber\\
&=&-\inf_{u}\{\langle v_{x}(s,x),A_sx+B_su\rangle+{1\over2}tr\big((C_sx+D_su)(C_sx+D_su)^{\ast}v_{xx}(s,x)\big)+\langle\psi_{x}(s,x),C_sx+D_su\rangle\nonumber\\
&&\ \ \ \ \ \ +\lambda_s v(s,x)+\langle Q_sx,x\rangle+\langle R_su,u\rangle\}ds+\psi(s,x)dW_s,
\end{eqnarray*}
which demonstrates that $(v,\psi)$ is the classical solution of the stochastic HJB equation.
\begin{flushright}
	\qed
\end{flushright}
\end{proof}
Having a classical solution of stochastic HJB equation. One can find the optimal control for the LQ problem.
\begin{prop}
	The optimal control of LQ problem is given by
	\begin{eqnarray*}
		u_s=-(R_s +D^{\ast}_sP_sD_s)^{-1}\big[P_sB_s+C^*_sP_sD_s+L_sD_s\big]^*X_s.
	\end{eqnarray*}
\end{prop}
\begin{proof}
	By Proposition \ref{prop_svt}, we see that the candidate $u$ for the optimal control is of the following feedback form:
	\begin{eqnarray*}
		u_s=-(R_s +D^{\ast}_sP_sD_s)^{-1}\big[P_sB_s+C^*_sP_sD_s+L_sD_s\big]^*X_s.
	\end{eqnarray*}
To show that it is indeed the optimal control, one only need to prove that it is a admissible control, which is proved in Tang \cite{tang2015dynamic}.
\begin{flushright}
	\qed
\end{flushright}
\end{proof}
Finally, applying It\^{o} formula to $dP_sX_sk_s$, we immediately have the desired relationship for LQ problem.
\begin{thm}\label{mz28}
For LQ Problem, we have the relationship between stochastic maximum principle and stochastic dynamical programming below:
\begin{eqnarray*}
p_s&=&-2P_sX_sk_s,\\
q_s&=&-2\big[P_s(C_sX_s+D_su_s)+L_sX_s\big]k_{s}.
\end{eqnarray*}
\end{thm}



\bibliographystyle{plain}

\end{document}